\documentclass[onefignum,onetabnum]{siamart171218}

\usepackage{lipsum}
\usepackage{amsfonts}
\usepackage{graphicx}
\usepackage{epstopdf}
\usepackage{algorithmic}
\ifpdf
  \DeclareGraphicsExtensions{.eps,.pdf,.png,.jpg}
\else
  \DeclareGraphicsExtensions{.eps}
\fi

\newsiamremark{remark}{Remark}
\newsiamremark{hypothesis}{Hypothesis}
\crefname{hypothesis}{Hypothesis}{Hypotheses}
\newsiamthm{claim}{Claim}

\headers{Improving the efficiency of a Hyperpower method}{A. Cordero, and J. R. Torregrosa}

\title{Improving the efficiency of a Hyperpower method for approximating inverse matrices\thanks{Submitted to the editors 2nd September 2026.
\funding{This research was partially supported by the European Research Council (ERC) via
Horizon Europe Advanced Grant, grant agreement number 101097688 (”PeroSpiker”).}}}

\author{Alicia Cordero\thanks{Instituto de Matemática Multidisciplinar, Universitat Politècnica de València, València, Spain
  (\email{acordero@mat.upv.es}).}
\and Juan R. Torregrosa\thanks{Instituto de Matemática Multidisciplinar, Universitat Politècnica de València, València, Spain
  (\email{jrtorre@mat.upv.es}).}}

\usepackage{amsopn}

\makeatletter
\newcommand*{\addFileDependency}[1]{% argument=file name and extension
  \typeout{(#1)}% latexmk will find this if $recorder=0 (however, in that case, it will ignore #1 if it is a .aux or .pdf file etc and it exists! if it doesn't exist, it will appear in the list of dependents regardless)
  \@addtofilelist{#1}% if you want it to appear in \listfiles, not really necessary and latexmk doesn't use this
  \IfFileExists{#1}{}{\typeout{No file #1.}}% latexmk will find this message if #1 doesn't exist (yet)
}
\makeatother

\begin{document}

\maketitle

% REQUIRED
\begin{abstract}
  In this manuscript, scalar norm-type accelerators are used for first time in order to increase the order of convergence of known iterative methods for solving the matrix equation $X^{-1}-A=0$, while decreasing the number of matrix-matrix products. Specifically, our proposed scheme upgrades the order of convergence in a $50\%$, achieving the sixth-order of convergence with the computational cost of fourth-order Hyperpower method, in terms of matrix-matrix multiplications. In the main results, not only the convergence and its order is proven, but also its efficiency and stability. This new technique provides good numerical results, even compared with existing fourth-, sixth- and eighth-order schemes in large matrices of up to $10^6$ entries, but also promises the opening of a new kind of procedures with good performance and scalability. Finally, an application to digital image restoration is made to show the behavior of the method, compared with existing ones.
\end{abstract}

% REQUIRED

% REQUIRED

\section{Introduction}
Computation of matrix inverses is a fundamental task in numerical linear algebra with applications spanning across signal processing, optimization, and modern deep learning. When dealing with large-scale matrices, direct solvers such as Gaussian elimination, QR decomposition, or singular value decomposition (SVD) often encounter performance problems because they are difficult to parallelize efficiently on modern hardware architectures, such as Graphics Processing Units (GPUs).

The most well-known scheme to estimate the inverse of a nonsingular complex matrix $A\in \mathbb{C}^{n\times n}$ is the second-order Newton-Schulz procedure (see \cite{NS}), with iterative expression
\begin{equation}\label{eq_NS}
    X_{k+1}=X_k(2I-AX_k), \quad k=0,1,\ldots,
\end{equation}
where $I$ is the identity matrix of size $n\times n$.

A comprehensive analysis of the stability, convergence, and properties of this and related matrix functions can be found in the text by Higham \cite{Higham}.

Although the classical Newton-Schulz method maps perfectly to the highly parallel capabilities of GPUs, its quadratic convergence rate requires a large number of total iterations until the norm of residual matrix $R_k = I - AX_k$ is small. To accelerate the convergence behavior, the iterative scheme can be generalized into a higher-order family known as the Hyperpower methods \cite{stanimirovic2016hyperpower}. For an arbitrary integer order of convergence $p \geq 2$, the standard Hyperpower method is defined as:
\begin{equation}\label{eq_HP}
    X_{k+1} = X_k \sum_{j=0}^{p-1} R_k^j = X_k(I + R_k + R_k^2 + \dots + R_k^{p-1}), \quad k=0,1,\ldots.
\end{equation}
Mathematically, this family achieves a $p$-th order of convergence. However, evaluating this polynomial in its standard form requires exactly $p$ matrix-matrix products per iteration (including the calculation of the residual matrix $R_k$). This linear growth in computational cost per step significantly degrades the overall computational efficiency for high values of $p$.

Consequently, a major focus of current research is to design alternative formulations or factorizations of the Hyperpower polynomial that try to minimize the number of required matrix-matrix products per iteration while preserving the high order of convergence \cite{soleymani2015class, krstic2018factorizations, kaur2020efficient}. For instance, specific factorizations have achieved fifth-order convergence for estimating Moore-Penrose inverses using only four matrix products \cite{kaur2020efficient}, or higher orders by cleverly grouping terms \cite{krstic2018factorizations}, but with the original Hyperpower method. Reducing the number of matrix products below the standard $p$-bound remains the primary strategy for improving the Ostrowski efficiency index \cite{ostrowski1960solution} of these algorithms.

In this paper, we propose a novel systematic method that improves upon the standard Hyperpower scheme by reducing the number of matrix-matrix products per iteration. Our approach lowers the traditional upper bound of $p$ products for order $p=6$ by using the novel approach of scalar accelerators, which has recently been very successful in the context of nonlinear systems of equations. As far as we now, this is the first time that this kind of technique is used in the context of inverse matrix estimation.

Through rigorous theoretical analysis and numerical experiments, we demonstrate that the proposed method maintains the expected high-order convergence rate while requiring less computational effort per step, making it highly competitive for modern parallel computing environments.

Given a nonlinear system of equations, $F(x)=0$, $F:\Omega\subset \mathbb{R}^n \to \mathbb{R}^n$, Singh, Sharma and Kumar designed in \cite{ssk} the fifth-order iterative method
\begin{eqnarray}\label{metodo_SNL}
    y^{(k)}&=&x^{(k)}-[F'(x^{(k)})]^{-1}F(x^{(k)}), \quad k=0,1,\ldots, \nonumber \\
    x^{(k+1)}&=&y^{(k)}-(1+\nu_k)[F'(y^{(k)})]^{-1}F(y^{(k)}),
\end{eqnarray}
where the first step is the Newton's method, the second one is a damped composition of Newton's scheme and the dumping parameter uses the scalar accelerator $\nu_k=\dfrac{\|F(y^{(k)})\|_2^2}{\|F(x^{(k)})\|_2^2}$.

Usually, the extension of scalar iterative methods on the inverse estimation problem is made by applying the scalar iterative expression to the nonlinear function $f(x)=\frac{1}{x}-a$, being $a$ a non-zero real parameter. In this case, the matrix iterative expression corresponding to the vectorial scheme \eqref{metodo_SNL} is not inverse-free, and we discard it. However, a small modification in the definition of the scalar accelerator $\nu_k$, taking into account the calculated residuals, makes it feasible. So, in the context of estimating the inverse of a complex nonsingular matrix $A$, we adapt the vectorial scheme \eqref{metodo_SNL} as
\begin{eqnarray}\label{metodo_inversa}
    Y_{k}&=&X_{k}\left( 2I-AX_k\right), \quad k=0,1,\ldots, \nonumber \\
    X_{k+1}&=&Y_{k}\left((2+\nu_k)I-(1+\nu_k)AY_{k}\right),
\end{eqnarray}
where $\nu_k=\dfrac{\|I-AY_k\|^2}{\| I-AX_k \|^2}$, and $\|\cdot\|$ is any matrix multiplicative norm. From now on, this scheme is denoted by CTM. Let us remark that, in the numerical section, we use the Frobenius norm.

The computational cost of this scheme per iteration is four matrix-matrix products and two Frobenius norm calculations, due to their lower cost. Taking into account that the best existing schemes of the same order of convergence need at least six matrix products per iteration, the improvement of the efficiency in this proposal is about $16\%$.

To fully understand the convergence of iterative methods, we can analyze them as real discrete dynamical systems. Within this framework, the iterative function is treated as a rational operator. Its stability is evaluated by studying the nature of its fixed points, critical points, and their corresponding basins of attraction. This dynamical approach helps to identify stable behaviors and determine which initial estimations lead to the desired root without entering chaotic regions.

This dynamical perspective has been successfully extended to matrix equations, particularly for estimating matrix inverses. Instead of using classical norm arguments, recent research analyzes the stability of inverse matrix solvers by modeling the iteration as a discrete dynamical system \cite{2024inverse, cordero2023improving}. By evaluating the matrix error or using eigenvalues, researchers construct a rational stability operator \cite{2025new}. Analyzing the stable fixed points, strange fixed points, and critical points of this operator allows the identification of real stability intervals. This technique provides the boundaries for the initial approximation $X_0$, guaranteeing that the matrix iteration converges safely and is robust to numerical perturbations.

The paper is organized as follows. Our main result is in Section \ref{sec:main}, where the convergence of our proposed scheme for estimating inverse matrices and its order of convergence are proven. Afterwards, its efficiency and stability are analyzed in Sections \ref{sec:eff} and \ref{sec:stability}, respectively. Some experimental results on large random matrices and a practical case on image processing appear in Section \ref{sec:experiments}, followed by the final conclusions in Section \ref{sec:conclusions}.

\section{Main results}
\label{sec:main}

In this section, we prove the convergence of the proposed method, establish the hypothesis on the initial estimations used and set the order of convergence.

\begin{theorem}\label{thm:bigthm}
  If $A \in \mathbb{C}^{n \times n}$ is nonsingular, and the initial estimation $X_0 \in \mathbb{C}^{n \times n}$ satisfies $\| I-AX_0 \|\leq \sigma<1$, then the iterative process CTM described in \eqref{metodo_inversa} converges to $A^{-1}$ with order of convergence six.
\end{theorem}

\begin{proof}
Let us define $R_k=I-AX_k$, and $R_k^Y=I-AY_k$, for $k=0,1,\ldots$. Then,
\begin{equation}\label{residuos}
R_k^Y=I-AY_k=I-AX_k(2I-AX_k)=(I-AX_k)^2=R_k^2,
\end{equation}
and therefore
\[
\nu_k=\frac{\| R_k^Y\|^2}{\|R_k\|^2} = \frac{\| R_k^Y\|^2}{\|R_k\|^2}  \leq  \frac{\| R_k\|^4}{\|R_k\|^2}=\|R_k\|^2.
\]

From the iterative expression \eqref{metodo_inversa}, it is deduced that
\begin{eqnarray}
    R_{k+1}&=& I-AY_{k}\left((2+\nu_k)I-(1+\nu_k)AY_{k}\right)\nonumber\\
            &=& I-(2+\nu_k)AY_k+(1+\nu_k)AY_kAY_k.
\end{eqnarray}
Taking into account that $AY_k=I-R_k^Y$, $R_{k+1}$ can be expressed in terms of $R_k^Y$ as follows
\begin{eqnarray}
    R_{k+1}&=& -(1+\nu_k)I+(2+\nu_k)R_k^Y+(1+\nu_k)(I-2R_k^Y+(R_k^Y)^2)\nonumber\\
            &=& -\nu_k R_k^Y+(1+\nu_k)(R_k^Y)^2\\
            &=& R_k^Y(-\nu_k I+(1+\nu_k)R_k^Y).\nonumber
\end{eqnarray}
By using the relation between both residuals \eqref{residuos}, it is clear that
\begin{equation}\label{eq_2.4}
R_{k+1}=-\nu_k R_k^2+(1+\nu_k)R_k^4.
\end{equation}
Taking norms and using the triangular inequality, the relation \eqref{residuos}, and the multiplicative properties of the norms, a bound of the residual $R_{k+1}$ can be found:
\begin{eqnarray}\label{recurrencia}
    \|R_{k+1}\|&=& \nu_k \|R_k^2\|+(1+\nu_k)\|R_k^4\|\nonumber\\
            &=& \frac{\| R_k^2\|^3}{\|R_k\|^2} +\left(1+\frac{\| R_k^2\|^2}{\|R_k\|^2}\right)\|R_k^4\|\\
            &\leq& 3\|R_k\|^4.\nonumber
\end{eqnarray}

By assuming that $R_0=\| I-AX_0 \|\leq \sigma<1$, \eqref{recurrencia} implies
\[
\|R_{1}\|\leq 3\|R_0\|^4 \leq 3 \sigma^3 \|R_0\|.
\]
The real root of $3 \sigma^3=1$ gives the upper bound of $\sigma$, $\sigma < \left(\frac{1}{3} \right)^{1/3}$. It can be proven by induction that $\|R_k\|\leq M \|R_{k-1}\|$, for all $k\in \mathbb{N}$, being $M=3\sigma^3<1$. Then,
\[
 \|R_{k+1}\|\leq M \|R_k\|\leq M^2 \|R_{k-1}\|\leq \cdots \leq M^{k+1} \|R_0\|.
\]
As $M<1$, when $k$ tends to infinity, $M^{k+1}$ tends to zero. Since $\|R_0\|$ is bounded, it is concluded that  $\|R_{k+1}\|$ tends to zero. So, $AX_k$ tends to the identity matrix $I$, and the convergence of $X_k$ to $A^{-1}$ is proven.

Let us now prove the order of convergence of the iterative scheme. We define the local error $e_k=A^{-1}-X_k$. It is easy to check that $Ae_k=R_k$. So, using \eqref{eq_2.4},
\begin{eqnarray}
    e_{k+1}&=&A^{-1}R_{k+1}\nonumber\\
        &=& A^{-1}\left( -\nu_k R_k^2+(1+\nu_k)R_k^4 \right)\\
        &=& -e_k\left( \nu_k Ae_k-(Ae_k)^3\right)+\nu_k e_k (Ae_k)^3.\nonumber
\end{eqnarray}

Let us remark that
\[
\nu_k Ae_k-(Ae_k)^3=-\nu_k R_k+R_k^3\leq \|R_k\|^2 R_k+R_k^3,
\]
that tends to zero with third-order of convergence (with fourth-order, as it is multiplied by $e_k$). So,
\begin{eqnarray}
   \| e_{k+1}\|&\leq&\|e_k\left( \nu_k Ae_k-(Ae_k)^3\right)\|+\|\nu_k e_k (Ae_k)^3\|\nonumber\\
        &=& \epsilon_k \|e_k\|^4+\|A\|^5\|e_k\|^6,
\end{eqnarray}
where $\epsilon_k$ tends to zero, Then, given $\epsilon>0$, there exists a $k_0\in \mathbf{N}$ such that for any $k\geq k_0$, $\epsilon_k\leq \epsilon$. Taking $\epsilon=\|e_k\|^2$, exist a $k_0'\in \mathbf{N}$ such that for any $k\geq k_0'$,
\[
\|e_{k+1}\|\leq \|e_k\|^6+\|A\|^5 \|e_k\|^6=(1+\|A\|^5)\|e_k\|^6.
\]
Therefore, the sixth-order of convergence of the method is proven.
\end{proof}

\section{Computational Efficiency Analysis}\label{sec:eff}

To evaluate the performance of the proposed method against standard algorithms, we use the Ostrowski efficiency index \cite{ostrowski1960solution, traub1964iterative}. The efficiency index $EI$ provides a theoretical balance between convergence speed and per-iteration computational cost. It is defined as:
\begin{equation}
    EI = p^{\frac{1}{M}},
\end{equation}
where $p$ represents the local order of convergence of the method, and $M$ denotes the number of matrix-matrix multiplications required per step, of order $O(n^3)$, being $n \times n$ the size of the matrix. Matrix-matrix products dominate the overall computational cost, meaning other operations such as matrix additions or scaling are considered negligible. The norm calculation is of order $O(n^2)$, in case of infinity or Frobenius norms, and higher in other cases. So, when large matrices are taken into account, the only cost to be considered is that of matrix-matrix multiplications.

In the standard Hyperpower family, achieving a $p$-th order of convergence requires exactly $M_{hp} = p $ matrix multiplications. Consequently, its efficiency index is given by:
\begin{equation}
    EI_{HPp} = p^{\frac{1}{p}}.
\end{equation}
As the order $p$ increases, $EI_{HPp}$ monotonically decreases. For instance, the classical second-order Newton-Schulz method ($p=2, M=2$) yields $EI_{NS} = 2^{1/2}\approx 1.4142 $, whereas a standard fourth- or sixth-order Hyperpower method ($p=4, M=4$,  or $p=6, M=6$, or $p=8, M=8$) holds NS one, $EI_{HP4} = 4^{1/4}= 2^{1/2} $, or drops to  $EI_{HP6} = 6^{1/6} \approx 1.348$, and $EI_{HP6} = 8^{1/8} \approx 1.2968$. This theoretical decline demonstrates that simply increasing the order without optimization degrades performance.

Our proposed algorithm breaks this limitation by reducing the cost to a lower bound, with $M_{CTM} = 4$. Therefore, the efficiency index of our method becomes:
\begin{equation}
    EI_{CTM} = 6^{\frac{1}{4}}\approx 1.5651.
\end{equation}
The efficiency indices for all these schemes are shown in Figure \ref{figura1}.
\begin{figure}[ht!!]
  \centering  \label{figura1}\includegraphics[width=.7\textwidth]{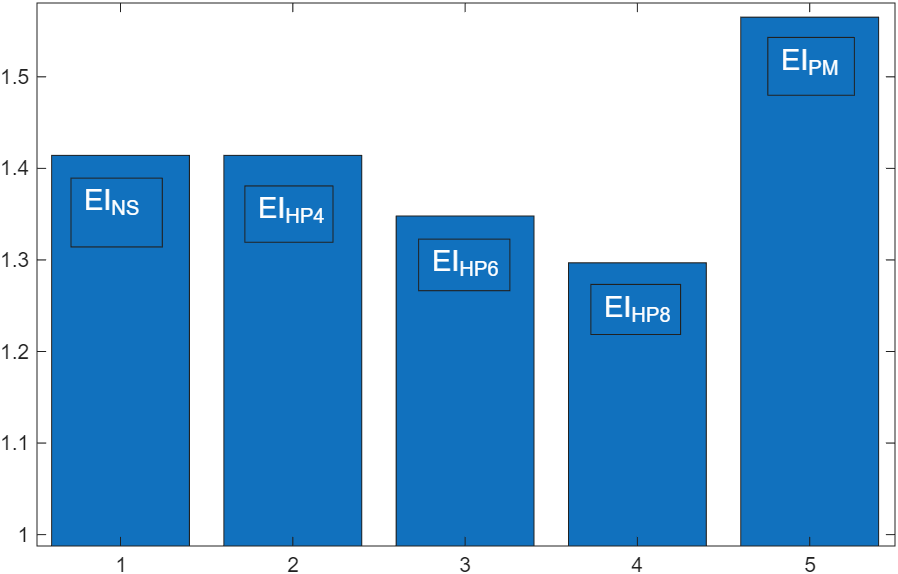}
  \caption{Efficiency indices of the proposed and known methods}
\end{figure}

This comparison shows that our formulation delivers a better trade-off between convergence rate and computational effort, establishing a higher theoretical efficiency index than Hyperpower schemes, which are, to our knowledge, the most efficient schemes to date.

\section{Stability analysis}\label{sec:stability}

Dynamical analysis is nowadays a standard approach for comparing iterative methods. Beyond measuring convergence speed, it reveals how a scheme behaves for different starting guesses. This also shows whether the method is stable and dependable. We briefly introduce the discrete-dynamics concepts used here. Readers seeking a deeper background may consult and \cite{Devaney}.

Assume $T:{\mathbb{R}}\rightarrow {\mathbb{R}}$ is a rational map built by running an iterative scheme on a low-degree polynomial or rational function $p(x)$. Starting from $x_{0}$, its orbit is the sequence
\begin{flalign*}
\left\lbrace x_{0},T(x_{0}),T^{2}(x_{0}),\ldots,T^{n}(x_{0}),\ldots\right\rbrace.
\end{flalign*}
A value $x^{*} \in { \mathbb{R}}$ is called $k$-periodic when applying $T$ repeatedly $k$ times returns to $x^{*}$, i.e., $T^k(x^{*})=x^{*}$. When $k=1$, we have a fixed point. The stability of a $k$-periodic cycle depends on the product of derivatives along the path. If $\lvert T'(x^{*})\cdot (T^2(x^{*}))' \cdots (T^{k-1}(x^{*}))'\rvert <1$, the cycle pulls nearby values toward it (attracting). If the product equals $0$, it strongly attracts (superattracting). A product greater than $1$ pushes values away (repulsive). A product equal to $1$ yields neutral or parabolic behavior. Furthermore, if a fixed point does not solve $p(x)=0$, we label it a strange fixed point.

A point $x_{c}$ is critical when the derivative vanishes, $T'(x_{c})=0$. When such a point is not a root of the original $p(x)$, it is called a free critical point.

The set of all starting values that eventually approach an attracting fixed or periodic point $x^*$ is the basin of attraction, written as
\begin{flalign*}
A(x^{*})=\{x_0\in\hat{ \mathbb{C}}:T^{n}(x_0)\rightarrow x^{*}, n\rightarrow\infty\}.
\end{flalign*}
A classic theorem by Fatou (\cite{Fatou}) and Julia (\cite{Julia}) connects these basins to critical points. Specifically, the immediate basin surrounding an attracting cycle must contain at least one critical point.

\begin{theorem}[Fatou-Julia] \label{theo0}
Let $T$ be a rational map. Every immediate basin around an attracting fixed or periodic point includes at least one critical point.
\end{theorem}

Next, we study the dynamics of the scheme \eqref{metodo_SNL}. To test how the method approximates a matrix inverse, we apply it to $f(x)=\dfrac{1}{x}-1$ instead of the usual quadratic test function.

Let $T$ act on the extended real line $D\subset\overline{\mathbb{R}}$, with $\overline{\mathbb{R}}=[-\infty,+\infty]$. Substituting $f(x)=\dfrac{1}{x}-1$ into \eqref{metodo_SNL} yields the rational expression
\begin{equation}\label{R}
T(x)= -((-2 + x) x (3 - 6 x + 7 x^2 - 4 x^3 + x^4)).
\end{equation}
The next result identifies all real fixed points of this map.

\begin{theorem}\label{theo5}
Consider the operator $T:D\subset \overline{\mathbb{R}}\to \overline{\mathbb{R}}$ given by \eqref{R}. The true solution $x=1$ is a superattracting fixed point. Additionally, $x=0$ and $x=\infty$ act as strange fixed points; the former repels nearby values, while the latter strongly attracts them.
\end{theorem}

\begin{proof}
Fixed points satisfy $T(x)=x$, which is equivalent to finding the roots of
\begin{equation*}
T(x)=  -((-2 + x) x (3 - 6 x + 7 x^2 - 4 x^3 + x^4))-x =0.
\end{equation*}
Solving this equation yields $x=0$ and $x=1$ as the only real solutions; hence, these are the fixed points. By examining the auxiliary map $In(x)=\dfrac{1}{T(1/x)}$, we see $In(0)=0$, which confirms that $x=\infty$ is also a fixed point.

The point $x=1$ is superattracting ($T'(1)=0$), as it solves $f(x)=1/x-1$, and the scheme \eqref{metodo_SNL} converges with order higher than two. To check the stability of the strange point $x=0$, we differentiate $T(x)$:
\begin{equation*}
T'\left( x \right) =-6 (-1 + x)^5.
\end{equation*}
Evaluating the magnitude at zero gives $\left|T'\left( 0\right) \right|=6>1$, so $x=0$ is a repulsor. A quick check shows $\left|In'\left( 0\right) \right|=0$, meaning $x=\infty$ is superattracting.
\end{proof}

Setting $T'(x)=0$ reveals that $x=1$ is a critical point, and $x=\infty$ serves as the only free critical point, and the method either converges to $x=1$ or diverges. This is due to Julia-Fatu Theorem, as there is no other critical point that creates any other basin of attraction.

To visualize the stability more clearly, we plot a dynamical line (a visualization tool first used in \cite{RD} for methods with memory). Figure \ref{figura2} displays this plot. It was created in Matlab R2024b using a $500 \times 500$ grid and a limit of $50$ steps. Each grid point starts the iteration at a different value.

\begin{figure}[ht!!!]
\centering
\includegraphics[width=0.7\textwidth]{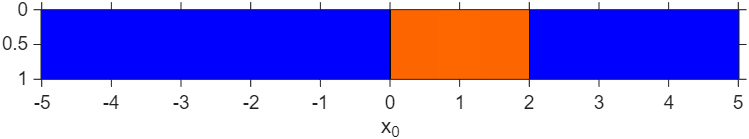}
\caption{Dynamical line for $R(x)$}
\label{figura2}
\end{figure}

If the iteration reaches the true root $x=1$ within $50$ steps and the error drops below $10^{-3}$, the grid point is marked orange. If it diverges toward infinity, it is marked blue. Figure \ref{figura2} clearly shows two basins of attraction: the orange area corresponds to convergence toward the solution $x=1$, and the blue area marks divergence.

\section{Experimental results}
\label{sec:experiments}

These numerical tests are performed on an Intel(R) Core(TM) i7-9700 CPU @3.00 GHz, with 32 GB of RAM, using MATLAB R2024b, with double-precision arithmetic and $\|I-AX_k\|<10^{-5}$ as a stopping criterion. The methods used for comparison with the proposed scheme CTM are the Newton-Schulz procedure \eqref{eq_NS} and the Hyperpower methods \eqref{eq_HP} of orders $p=4$, $p=6$, and $p=8$ (denoted by HP4, HP6, and HP8, respectively). The initial estimation used for each scheme is $X_{0}=\dfrac{A^{T}}{\| A \|^{2}}$. In all the calculations, the norm $\|\cdot\|$ used is the Frobenius norm. Moreover, we present in the tables the number of iterations needed to get convergence, the last calculated residual $\|I-AX_{k+1}\|$, and the numerical estimation $coc$ of the order of convergence $p$ given by Jay in \cite{jay2001note},
\[
p\approx coc=\frac{\ln\left(\frac{\|I-AX_{k+1}\|}{\|I-AX_k\|}\right)}{\ln\left(\frac{\|I-AX_{k}\|}{\|I-AX_{k-1}\|}\right)}.
\]
Let us remark that coc is a vector defined using the residuals obtained during the iterative process. The value appearing in the table is the last stable one, when it is stabilized. If it is not stable, the symbol '-' is used in the table. %%%%%%%%%%%%%%%%%%%%%%%%

Moreover, the execution time is calculated using the commands \texttt{tic}-- \texttt{toc}, as the mean of $100$ consecutive executions of the code. It appears in the table as Time, expressed in seconds.

Firstly, we work with a random $300\times 300$ matrix (defined by using the command \texttt{rand} of Matlab), and its inverse is estimated using all the iterative schemes mentioned. The obtained results are shown in Table  \ref{tabla2}. %%%%%%%%%%%%%%%%

% \begin{table}[ht!!]
% \centering
% \caption{Comparison of iterative methods for inverse estimation of a random $A\in \mathbb{C}^{300\times 300}$ using $\|\cdot\|_\infty$ norm}\label{tabla1}
% \begin{tabular}{lcccc}
% \hline
% Method & Iter & $\|I-AX_{k+1}\|$ & $coc$  & Time (s)\\
% \hline
% NS & 32 & $8.14\times 10^{-06}$ & 2.00  & 0.01662 \\
% HP4 & 17 & $2.91\times 10^{-11}$ & 3.92  & 0.02066 \\
% HP6 & 13 & $6.86\times 10^{-06}$ & 3.78  & 0.02538 \\
% HP8 & 12 & $6.39\times 10^{-13}$ & 6.62  & 0.03788 \\
% PM  & 14 & $6.48\times 10^{-13}$ & 5.29  & 0.01394 \\
% \hline
% \end{tabular}
% \end{table}

\begin{table}[ht!!]
\centering
\caption{Comparison of iterative methods for inverse estimation of a random $A\in \mathbb{C}^{300\times 300}$}\label{tabla2} %%%%%%%%%%%%%%%%%%%%%%%%%%%%%%
\begin{tabular}{lcccc}
\hline
Method & Iter & $\|I-AX_{k+1}\|$ & $coc$  & Time (s)\\
\hline
NS & 33 & $1.22\times 10^{-10}$ & 2.00  & 0.01513 \\
HP4 & 17 & $1.22\times 10^{-10}$ & 3.46  & 0.01760 \\
HP6 & 13 & $9.44\times 10^{-06}$ & 3.61  & 0.02309 \\
HP8 & 12 & $6.68\times 10^{-13}$ & 5.71  & 0.03607 \\
CTM  & 14 & $9.20\times 10^{-12}$ & 4.40  & 0.01241 \\
\hline
\end{tabular}
\end{table}

In Table \ref{tabla2}, the lowest mean convergence time is achieved by the proposed method CTM, even compared with higher-order schemes such as HP8. This is a significant result due to the lowest number of matrix-matrix products per iteration in the CTM scheme.

Now, we test the performance of the methods when estimating the inverse of a random matrix $B$ of size $1000\times 1000$. The obtained results are presented in Table \ref{tabla4}.

% \begin{table}[ht!!]
% \centering
% \caption{Comparison of iterative methods for inverse estimation of a random $B\in \mathbb{C}^{1000\times 1000}$ using $\|\cdot\|_\infty$ norm}\label{tabla3}
% \begin{tabular}{lcccc}
% \hline
% Method & Iter & $\|I-AX_{k+1}\|$ & $coc$ & Time (s)\\
% \hline
% NS  & 38 & $2.74\times 10^{-07}$ & 2.00  & 0.73910 \\
% HP4 & 20 & $2.33\times 10^{-12}$ & 4.00  & 0.90894 \\
% HP6 & 16 & $2.35\times 10^{-12}$ & 6.00  & 1.33252 \\
% HP8 & 14 & $2.38\times 10^{-12}$ & 8.00  & 1.91229 \\
% PM  & 16 & $2.40\times 10^{-12}$ & 6.00  & 0.66368 \\
% \hline
% \end{tabular}
% \end{table}

\begin{table}[ht!!]
\centering
\caption{Comparison of iterative methods for inverse estimation of a random $A\in \mathbb{C}^{1000\times 1000}$ }\label{tabla4} %%%%%%%%%%%%%%%%%%%%%%%%%%%%%%%%%%%%%%%%%%%%%%
\begin{tabular}{lcccc}
\hline
Method & Iter & $\|I-AX_{k+1}\|$ & $coc$ &  Time (s)\\
\hline
NS  & 38 & $1.21\times 10^{-06}$ & 2.00  & 0.79514 \\
HP4 & 20 & $3.05\times 10^{-12}$ & 4.00  & 0.86381 \\
HP6 & 16 & $2.68\times 10^{-12}$ & 6.00  & 1.40004 \\
HP8 & 14 & $2.66\times 10^{-12}$ & 7.91  & 1.81183 \\
CTM  & 16 & $6.90\times 10^{-10}$ & 5.98  & 0.70140 \\
\hline
\end{tabular}
\end{table}

The results appearing in Table \ref{tabla4} show a similar performance, with the lowest mean computational time for CTM scheme, and small differences in the residuals, being the number of iterations consistent with the order of convergence of the methods. In this case, the estimation of the methods' order of convergence given by coc is more precise.

Table \ref{tabla6} shows the results obtained by approximating the inverse of a benchmark set of Lehmer, Riemann, Hankel, and Leslie matrices, implemented in MATLAB inside \texttt{gallery}, with size $500\times 500$, in all cases. These example matrices cover a wide range of cases: Lehmer matrices $A=(a_{i,j})$, whose entries are calculated as
\[
a_{i,j}=\left\{
\begin{array}{cc}
  i/j, & j\geq i, \\
  j/i, & j<i,
\end{array}
\right.
\]
are symmetric, invertible matrices whose trace is equal to the number of rows/coulmns of the matrix, and they are used in signal processing. Riemann matrices are complex symmetric invertible matrices, defined as
\[
a_{i,j}=\left\{
\begin{array}{cc}
  i-1, & \mbox{ if } i \mbox{ is a divisor of } j, \\
  -1, & otherwise.
\end{array}
\right.
\]
Hankel matrices are also symmetric, whose entries satisfy $a_{i,j}=a_{i+k,j-k}$, for $k=0,1,\ldots,j-1$, if $i\leq j$. Hilbert matrices are particular cases of Hankel matrices. They are defined in Matlab as
\[
a_{i,j}=\frac{0.5}{n-i-j+1.5},
\]
where $n$ is the size of the matrix. Finally, Leslie matrix appears in populations models, and its shape as $n\times n$ matrix is
\[
\begin{pmatrix}
  f_1 & f_2 & f_3 & \cdots & f_{n-1} & f_n \\
  s_1 & 0 & 0 & \cdots & 0 & 0 \\
  0 & s_2 & 0 & \cdots & 0 & 0 \\
  \vdots & \vdots & \vdots &  & \vdots & \vdots \\
  0 & 0 & 0 & \cdots & s_{n-1} & 0
\end{pmatrix}.
\]
It is an invertible matrix if $f_n\neq 0$ and $s_i\neq 0$ for all $i=1,2,\ldots,n-1$. In Matlab, all the entries $f_i$ and $s_i$ are defined as $1$.

\begin{table}[ht!!]
\centering
\caption{Comparison of iterative methods for inverse estimation of $500\times 500$ Lehmer, Riemann, Hankel, and Leslie matrices}\label{tabla6} %%%%%%%%%%%%%%%%%%%%%%%%%%%%%%%%%%%%%%%%%%%5
\begin{tabular}{lccccc}
\hline
Method & Matrix & Iter & $\|I-AX_{k+1}\|$ & $coc$ &  Time (s)\\
\hline
NS  & Lehmer & 41 & $1.34\times 10^{-06}$ & 1.82  &  0.12605 \\
HP4 & Lehmer & 21 & $1.27\times 10^{-06}$ & 3.47  &  0.13393 \\
HP6 & Lehmer & 17 & $1.61\times 10^{-11}$ &    -  &  0.19801 \\
HP8 & Lehmer & 15 & $1.44\times 10^{-11}$ & -     &  0.28908 \\
CTM  & Lehmer & 17 & $6.60\times 10^{-07}$ & -     &  0.09994 \\
\hline %\hline
NS  & Riemann & 36 & $2.22\times 10^{-09}$ & 2.00  & 0.11212 \\
HP4 & Riemann & 19 & $5.22\times 10^{-13}$ & 4.00  & 0.12157 \\
HP6 & Riemann & 15 & $1.62\times 10^{-12}$ & 6.00  & 0.17308 \\
HP8 & Riemann & 13 & $1.56\times 10^{-12}$ & 8.00  & 0.24414 \\
CTM  & Riemann & 15 & $2.68\times 10^{-10}$ & 6.00  & 0.08818 \\
\hline %\hline
NS  & Hankel & 18 & $4.92\times 10^{-08}$ & 2.00  & 0.05427 \\
HP4 & Hankel & 10 & $2.08\times 10^{-14}$ & 4.10  & 0.06238 \\
HP6 & Hankel & 8  & $1.40\times 10^{-14}$ & -     & 0.08733 \\
HP8 & Hankel & 7  & $1.46\times 10^{-14}$ & -     & 0.12175 \\
CTM  & Hankel & 8  & $1.39\times 10^{-09}$ & -     & 0.04623 \\
\hline %\hline
NS  & Leslie & 24 & $5.26\times 10^{-08}$ & 2.00  & 0.07523 \\
HP4 & Leslie & 13 & $1.15\times 10^{-12}$ & 4.00  & 0.08160 \\
HP6 & Leslie & 10 & $1.80\times 10^{-09}$ & 6.00  & 0.11076 \\
HP8 & Leslie & 9  & $2.41\times 10^{-13}$ & 8.00  & 0.15942 \\
CTM  & Leslie & 10 & $5.58\times 10^{-06}$ & 6.00  & 0.05717 \\
\hline %\hline
\end{tabular}
\end{table}

From Table \ref{tabla6}, we deduce that HP8 method requires the fewest iterations to converge, whereas the NS method is the slowest across all test matrices, as it can be inferred by their order of convergence. The CTM method matches the iteration count of HP6. However, in terms of execution time, CTM is the fastest method for all matrices because it has a lower computational cost per iteration. Conversely, higher-order methods like HP6 and HP8 require more CPU time due to the expensive matrix multiplications. Regarding the precision, high-order HP methods achieve the highest accuracy, reducing the residual error down to $10^{-14}$. The CTM method provides an intermediate level of numerical accuracy.

\subsection{Application to digital image restoration}

This section describes the numerical framework used to evaluate the performance of high-order iterative methods for matrix inversion, focusing on the realistic scenario of digital image restoration to remove motion or optical blur. In digital image processing, the horizontal degradation of a sharp two-dimensional image matrix due to optical blurring can be modeled as a matrix linear system:
\[
Y = X A^T + V,
\]
where $Y \in \mathbb{R}^{n \times n}$ represents the observed blurry image matrix, $A \in \mathbb{R}^{n \times n}$ is the blurring operator, $X \in \mathbb{R}^{n \times n}$ is the unknown sharp image matrix to be recovered, and $V \in \mathbb{R}^{n \times n}$ is the additive noise matrix. In this application, we consider $n=200$.

Directly solving this system by computing $A^{-1}$ is not feasible in practice because the matrix $A$ is highly ill-conditioned, meaning its singular values decay rapidly toward zero. As a consequence, direct inversion amplifies the noise matrix $V$, resulting in severely corrupted restorations. To overcome this limitation, we apply Tikhonov regularization, which stabilizes the inversion by solving a penalized least-squares problem, leading to the regularized system matrix $A_{\text{reg}} = A^T A + \lambda I$, where $\lambda > 0$ is the regularization parameter and $I$ is the $n \times n$ identity matrix. The matrix $A_{\text{reg}}$ is symmetric and positive definite (SPD), which ensures that all its eigenvalues are real and strictly positive, establishing a well-defined domain for iterative convergence.

To put this approach into practice, our experiment uses a specific matrix $X = X_{\text{circ}} \in \mathbb{R}^{n \times n}$ representing a Siemens-like target of concentric rings designed to test resolution limits. The construction of this matrix requires establishing a continuous spatial coordinate system over a grid. We define a symmetric spatial meshgrid spanning from $-10$ to $10$ in both dimensions, mapping each discrete pixel index $(i,j)$ to a continuous coordinate $(u_i, v_j)$. For every coordinate pair, the Euclidean distance from the origin is calculated as $r_{i,j} = \sqrt{u_i^2 + v_j^2}$. The image matrix values are then generated by evaluating a rapidly oscillating sinusoidal function governed by the squared radius, expressed as $S_{i,j} = \sin(r_{i,j}^2)$. To enforce maximum contrast and a distinct binary structure, a logical threshold is applied such that $X_{\text{circ}, i,j} = 1$ if $S_{i,j} > 0$ and $X_{\text{circ}, i,j} = 0$ otherwise. Because the squared radius compresses the wavelength as $r$ approaches the origin, the spacing between the concentric rings decreases rapidly toward the center, representing very high spatial frequencies. This image appears in Figure \ref{original_circ}.
\begin{figure}[ht!!!]
  \centering
  \includegraphics[width=.7\textwidth]{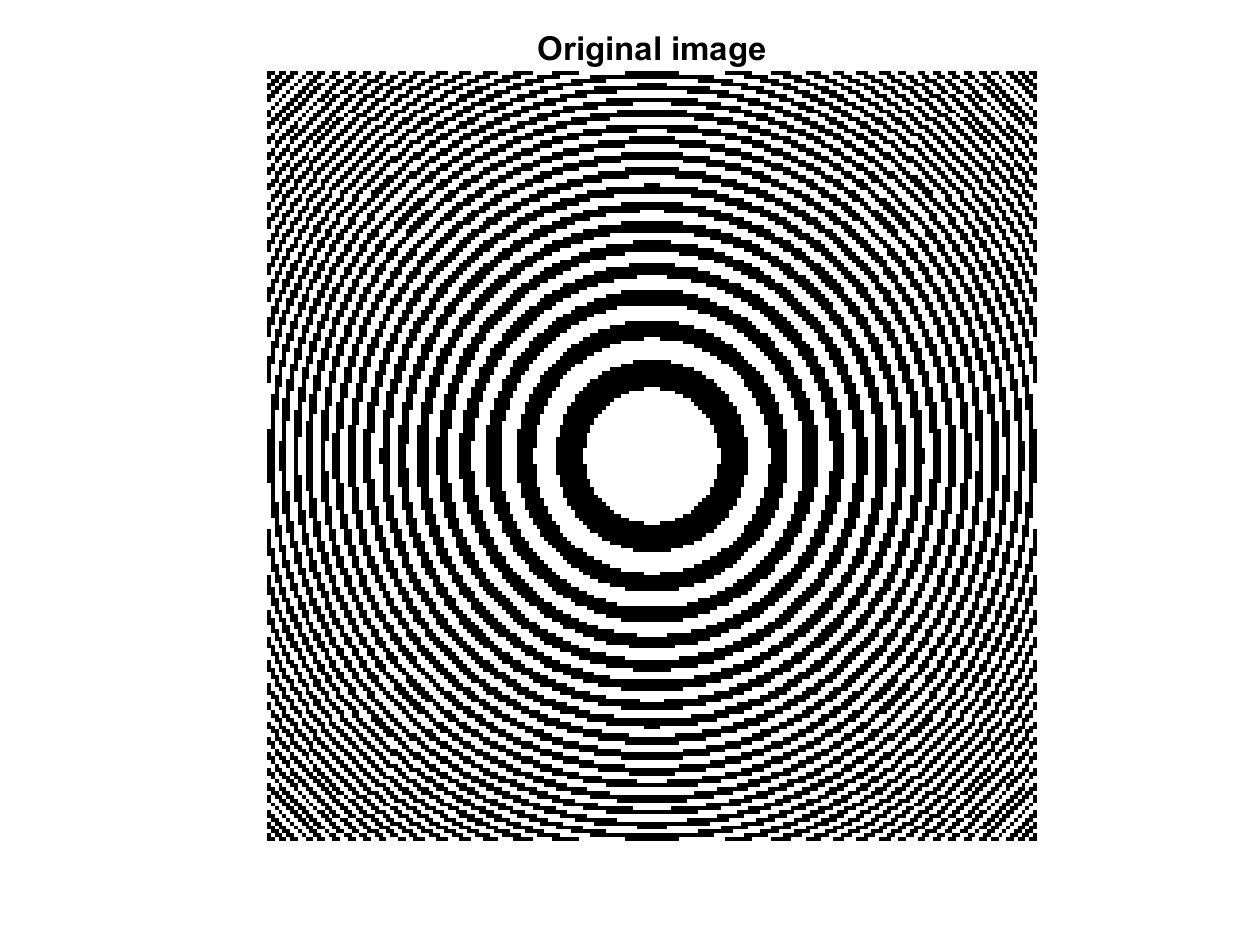}
  \caption{Original clean image without blurring}\label{original_circ}
\end{figure}

The choice of the regularization parameter $\lambda$ plays a critical role in the final matrix multiplication and determines the visual fidelity of the restored image. From a numerical perspective, the original blurring matrix $A$ is inherently ill-conditioned, meaning its smallest singular values are clustered arbitrarily close to zero. Without regularization ($\lambda = 0$), the iterative method would attempt to approximate the inverse of $A^T A$, causing the inverted small singular values to act as massive amplification factors. The insertion of the Tikhonov penalty term $\lambda I$ addresses this instability by artificially shifting the entire eigenspectrum of the system matrix upward by a factor of $\lambda$. This bound guarantees that the maximum noise amplification is strictly limited to $1/\lambda$, shielding the computation from high error propagation. Furthermore, this regularization transforms $A_{\text{reg}}$ into a strictly well-conditioned, symmetric positive definite operator and ensures that the initial residual matrix satisfies the strict spectral radius condition:
\[
\rho(I - A_{\text{reg}} X_0) < 1.
\]

To simulate the physical process of optical blurring in a controlled numerical environment, the operator $A \in \mathbb{R}^{n \times n}$ is constructed as a symmetric, structured Toeplitz matrix representing a space-invariant blur where every pixel is affected by the same distortion pattern. We generate a discrete Gaussian blur kernel vector $c \in \mathbb{R}^n$, which describes how light spreads from a single point, where each element is calculated as:
\[
c_i = \exp(-(i-1)^2 / (2\sigma^2)), \quad i = 1, 2, \ldots, n,
\]
with $\sigma=5$ controlling the width and severity of the blur. Using this vector as the first column and row, the full blurring matrix is assembled such that $A_{i,j} = c_{|i-j|+1}$, resulting in a banded matrix where the highest weights are concentrated along the main diagonal. The deterministic degradation of the two-dimensional digital image $X$ is executed by applying this operator row by row via the transformation $Y_{\text{blur}} = (A X^T)^T$, which is algebraically identical to $X A^T$. Once the perturbation matrix $V$ is added, this process mathematically blends the intensity of each pixel with its neighboring pixels along the same horizontal line, effectively erasing sharp transitions and high-frequency details. Reconstructing these inner details near the center requires a significantly higher level of numerical accuracy from the approximated inverse matrix.

Let $X_0 \in \mathbb{R}^{n \times n}$ be an initial approximation of $A_{\text{reg}}^{-1}$. The convergence of the algorithm depends strictly on the condition that the spectral radius satisfies $\rho(I - A_{\text{reg}} X_0) < 1$, which is guaranteed for the SPD matrix $A_{\text{reg}}$ by initializing the algorithm with:
\[
X_0 = \dfrac{A_{\text{reg}}^T}{ \|A_{\text{reg}}\|^2}.
\]
At each iteration $k$, the residual matrix is computed as $R_k = I - A_{\text{reg}} X_k$, and the iterative method updates the inverse approximation. Once the sequence converges to a final approximation $X_{\text{final}} \approx A_{\text{reg}}^{-1}$, the restored image matrix is obtained by computing:
\[
X_{\text{restored}} = Y A X_{\text{final}}.
\]
In this experiment, we force the iterative methods to make a maximum number of 10 iterations and calculate the residual at the last iteration, as well as the time required for each method. In this way, we compare the performance of the schemes in a limited framework where the most efficient methods must show convergence in small computational time.

Let us remark now the reason behind the preference for high-order iterative methods over classical direct factorization techniques, such as LU or Cholesky decomposition. While direct solvers are highly robust for small systems, they rely on sequential elimination and substitution phases that are inherently difficult to distribute across modern concurrent computing hardware. In contrast, the update engine of the iterative methods, such as Hyperpower schemes and our proposed method, relies strictly on successive matrix-matrix multiplications. This purely algebraic formulation maps perfectly onto the highly parallel architecture of modern Graphics Units (GPUs) and specialized tensor cores, enabling massive acceleration that direct solvers cannot achieve. Furthermore, direct factorization algorithms suffer from the structural phenomenon known as fill-in, which populates originally empty entries with non-zero values, thereby expanding the memory footprint and breaking the banded efficiency of the Toeplitz system.

In Figure \ref{recuperadas}, we observe the blurred images and the different recovered images by the different iterative methods.
\begin{figure}[ht!!!]
  \centering
  \includegraphics[width=1\textwidth]{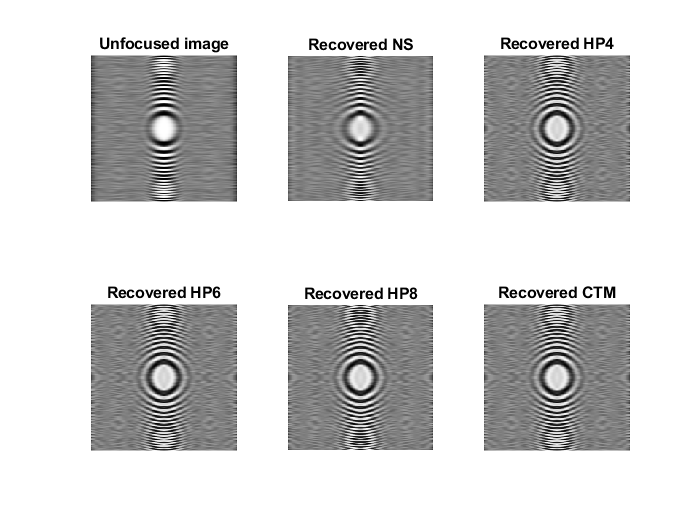}
  \caption{Blurred and recovered images, using different iterative schemes}\label{recuperadas}
\end{figure}
We observe that methods with order higher than two, succeed to recover enough the shape of the original image, without remarkable differences among them, from the graphical point of view. However, there are numerical differences to be noticed. In Table \ref{residuos_circ}, we can see the last residual and the execution time of all the iterative methods used, after the maximum of 10 iterations. We notice that NS and HP4 methods have not converged, and the lowest time among the converging schemes corresponds to the proposed method CTM.
\begin{table}[ht!!!]
  \centering
  \caption{Last residual and execution time for all the methods in image restoration problem}\label{residuos_circ}
  \begin{tabular}{lcc}
\hline
Method &  $\|I-AX_{k+1}\|_F$  & Time (s)\\
\hline
NS   & $13.360$              &  0.0049 \\
HP4  & $9.1654$              &  0.0089 \\
HP6  & $2.55\times 10^{-05}$ &  0.0118 \\
HP8  & $1.93\times 10^{-13}$ &  0.0175 \\
CTM   & $1.12\times 10^{-04}$ &  0.0086 \\
\hline
\end{tabular}
\end{table}

\section{Conclusions}
\label{sec:conclusions}

In this manuscript, a sixth-order method for approximating matrix inverses using only four matrix products has been designed, for the first time. This involves an improvement of $50\%$ in the order of convergence, for the same order, respect the sixth-order Hyperpower scheme. In the comparison with existing methods, it has proven to be efficient, stable and with good numerical results, improving the computational time needed by Hyperpower methods of orders 4, 6 and 8. Also tests on image processing provide very competent behavior. This kind of methods, using scalar norm-type acceleration parameters are the first in a new class of iterative methods that promise to expand the applicability of iterative methods on problems related with inverse matrix estimation.

%\section*{Acknowledgments}
%We would like to acknowledge the assistance of volunteers in putting together this example manuscript and supplement.

\bibliographystyle{siamplain}
\bibliography{references}

\end{document}